\documentclass[12pt,reqno]{amsart}
\usepackage{hyperref,cleveref,amssymb,eqnarray,mathrsfs,xcolor, amsmath}

\theoremstyle{plain}
\newtheorem{theorem}{Theorem}[section]
\newtheorem{proposition}[theorem]{Proposition}
\newtheorem{corollary}[theorem]{Corollary}
\newtheorem{lemma}[theorem]{Lemma}

\theoremstyle{definition}

\newtheorem{example}[theorem]{Example}

\newcommand{\Pelc}[1]{\left(\Sigma X_n\right)_p}
\newcommand{\goeso}{\xrightarrow{\mathrm{o}}}	

\newcommand{\RR}{\mathbb{R}}

\newcommand{\NN}{\mathbb{N}}

\DeclareSymbolFont{bbold}{U}{bbold}{m}{n}
\DeclareSymbolFontAlphabet{\mathbbold}{bbold}

\DeclareMathOperator{\sgn}{sgn}
\DeclareMathOperator{\Span}{span}

\renewcommand{\le}{\leqslant}
\renewcommand{\ge}{\geqslant}

\begin{document}

\title[Positive automorphisms of algebras of operators]{On positive automorphisms of algebras of operators on atomic Archimedean vector lattices}
\author{G. Cigler}
\address{Faculty of Mathematics and Physics, University of Ljubljana,
  Jadranska 19, 1000 Ljubljana, Slovenia \ \ \ and \ \ \
  Institute of Mathematics, Physics, and Mechanics, Jadranska 19, 1000 Ljubljana, Slovenia}
\email{gregor.cigler@fmf.uni-lj.si}

\author{M. Kandi\'c}
\address{Faculty of Mathematics and Physics, University of Ljubljana,
  Jadranska 19, 1000 Ljubljana, Slovenia \ \ \ and \ \ \
  Institute of Mathematics, Physics, and Mechanics, Jadranska 19, 1000 Ljubljana, Slovenia}
\email{marko.kandic@fmf.uni-lj.si}


\date{\today}

\begin{abstract}
Let $X$ be an Archimedean vector lattice. We investigate subalgebras of $\mathscr{L}(X)$ consisting of regular operators that contain all rank-one operators of the form $a \otimes \varphi_b$, where $a$ and $b$ are atoms of $X$ and $\varphi_b$ denotes the coordinate functional associated with $b$. Our main result shows that every positive automorphism of such a subalgebra contained in $\mathscr{L}(c_{00}(\Lambda))$, is necessarily spatial, meaning that it is implemented by a transformation of the form
$$
T \mapsto P D\, T\, D^{-1} P^{-1},
$$
where $P$ is a permutation operator and $D$ is a positive diagonal operator. An important tool for this analysis—one that is also of independent interest—is the Kakutani representation theorem, which we use to establish that every finite-dimensional vector subspace of $X$ is order closed.
\end{abstract}

\maketitle

\section{Introduction}\label{Section: Introduction}

Let $F$ be a field and let $\Phi\colon M_n(F)\to M_n(F)$ be an algebra automorphism of the algebra $M_n(F)$ of all $n\times n$ matrices over $F$. A classical result from linear algebra states that $\Phi$ is inner; that is, there exists an invertible matrix $A\in M_n(F)$ such that
$$
\Phi(T)=A T A^{-1}
$$
for all $T\in M_n(F)$. This follows readily from the Noether--Skolem theorem (see \cite[Theorem~3.14]{FD93}), which asserts that every automorphism of a finite-dimensional central simple $F$-algebra is inner.

In 1940, Eidelheit established the Banach space analogue (see \cite{Eid40}), proving that for a Banach space $X$ every algebra automorphism of the Banach algebra $\mathscr{B}(X)$ is inner. Later, Sourour extended the investigation to the Banach lattice setting (see \cite[Theorem~2]{Sou86}), showing that if $X$ is a Banach lattice, then every order-preserving algebra automorphism of the algebra $\mathscr{L}_r(X)$ of regular operators is inner and is implemented by a lattice isomorphism.

The present paper appears to be the first systematic study of order automorphisms of the algebra $\mathscr{L}_r(X)$ in the context of normed lattices $X$ that are not assumed to be complete. Somewhat unexpectedly, even the purely atomic case already exhibits substantial subtleties. 
In our main result \Cref{reprezentacija avtomorfizma} we assume that $X$ is an Archimedean atomic vector lattice and we provide a description of positive algebra automorphisms for subalgebras of $\mathscr L_n(X)$ containing operators $a\otimes \varphi_b$ for all atoms $a$ and $b$ from $X$.
In the special case where $X=c_{00}(\Lambda)$ is the vector lattice of finitely supported functions on a nonempty set $\Lambda$, we further obtain that every order-preserving algebra automorphism of $\mathscr{L}(c_{00}(\Lambda))$ or of $\mathscr{B}(c_{00}(\Lambda))$ is inner.

This paper is organized as follows. In \Cref{preliminaries} we collect the preliminary material used throughout the paper. In \Cref{Section: order closedness} we apply the Kakutani representation theorem to prove that every finite-dimensional vector subspace of an Archimedean vector lattice is order closed. Although we require this only in the special case of one-dimensional subspaces, the result is of independent interest, as it contributes to the broader understanding of order convergence and order-closed sets in vector lattices. This result will be used in \Cref{operators_atomic_VL}, where we study algebras of regular operators on Archimedean vector lattices generated by operators of the form $a\otimes \varphi_b$, with $a$ and $b$ atoms and $\varphi_b$ the coordinate functional associated with $b$.

In \Cref{automorphisms of algebras} we establish \Cref{reprezentacija avtomorfizma}, the main theorem of the paper. In \Cref{c00 - operator theoretical} we first obtain an operator-theoretic characterization of vector lattices of the form $c_{00}(\Lambda)$ for some nonempty set $\Lambda$. In particular, we show that an Archimedean vector lattice $X$ is lattice isomorphic to $c_{00}(\Lambda)$ if and only if
$$
\mathscr{L}(X,Y)=\mathscr{L}_b(X,Y)=\mathscr{L}_n(X,Y)
$$
for every Archimedean vector lattice $Y$. When $Y=c_{00}(\Lambda)$, this characterization allows us to rewrite \Cref{reprezentacija avtomorfizma} explicitly in that setting, yielding a concrete description of order-preserving automorphisms of the corresponding operator algebras.

Finally, in \Cref{Concluding} we provide concluding remarks concerning order-bounded functionals and the order dual of the lexicographic product of vector lattices.

\section{Preliminaries}\label{preliminaries}

Let $X,Y$ be vector lattices. A subset of $X$ is \emph{order bounded} if it is contained in some interval. 
By $\mathscr L(X,Y)$, $\mathscr L_b(X,Y)$ and $\mathscr L_n(X,Y)$ we denote the vector space of all linear operators, the space of all order bounded operators and  the space of all order continuous operators from $X$ to $Y$, respectively. If $X$ and $Y$ are normed lattices, then the space of all bounded operators is denoted by $\mathscr B(X,Y)$. If $X=Y$ we write $\mathscr L(X)$ instead of $\mathscr L(X,Y)$. Similarly, we define the symbols $\mathscr L_b(X)$, $\mathscr L_n(X)$ and $\mathscr B(X)$.
If $Y=\mathbb R$, we write 
$X':=\mathscr L(X,\mathbb R)$, $X^\sim:=\mathscr L_b(X,\mathbb R)$ and $X_n^\sim:=\mathscr L_n(X,\mathbb R)$. If $X$ is a normed lattice, we write $X^*:=\mathscr B(X,\mathbb R)$. The proof of the following well-known lemma is included for the sake of completeness.  

\begin{lemma}\label{order bounded operator functional}
Let $u$ be a nonzero vector of an Archimedean vector lattice $X$. For a linear functional $\varphi \in X'$ consider the rank one operator $u\otimes \varphi$.
 \begin{enumerate}
 \item $u\otimes \varphi$ is order bounded if and only if $\varphi$ is order bounded.
 \item $u\otimes \varphi$ is order continuous if and only if $\varphi$ is order continuous.
 \end{enumerate}
\end{lemma}

\begin{proof}
(i) Suppose first that $\varphi$ is order bounded and choose any interval $[x,y]\subseteq X$. There exists a scalar $\lambda>0$ such that
$\varphi([x,y])\subseteq [-\lambda,\lambda]$. Hence, for any $z\in [x,y]$ we have
$$|(u\otimes \varphi)(z)|=|\varphi(z)u|= |\varphi(z)| \,|u| \leq \lambda |u|$$ which proves that $u\otimes \varphi$ is order bounded.

Assume now that $u\otimes \varphi$ is order bounded and choose any interval $[x,y]\subseteq X$. There exists a positive vector $w$ such that
$(u\otimes \varphi)([x,y]) \subseteq [-w,w]$. Hence, for any $z\in [x,y]$ we have
$|\varphi(z)|\,|u|=|\varphi(z)u|=|(u\otimes\varphi)(z)|\leq w$. Since $X$ is Archimedean, the set $\varphi([x,y])$ needs to be bounded.

(ii) Suppose now that $x_\alpha \goeso 0$ in $X$. If $\varphi$ is order continuous, then $\varphi(x_\alpha)\to 0$. Since $X$ is Archimedean, we have
$(u\otimes \varphi)(x_\alpha)=\varphi(x_\alpha)u \goeso 0$. Conversely, suppose that $u\otimes \varphi$ is order continuous. From $\varphi(x_\alpha)u=(u\otimes \varphi)(x_\alpha)\goeso 0$ it follows $\varphi(x_\alpha)\to 0$.
\end{proof}

A positive vector $a$ of a vector lattice $X$ is called an \emph{atom} whenever it follows from $0\leq x, y\leq  a$ and $x\perp y$ that $x=0$ or $y=0$. The band $A$ generated by all atoms of $X$ is called the \emph{atomic part} of $X$. If $X=A$, then $X$ is an \emph{atomic vector lattice}. If $X\neq A$, then $C:=A^d$ is the \emph{continuous part} of $X$. By \cite[Theorem 1.36]{Aliprantis:06} the ideal $A+C$ is always order dense in $X$. If $X$ is Archimedean and not atomic, then $C^d=A^{dd}=A\neq X$, so that $C\neq \{0\}$. If the principal ideal $I_a$ generated by $a$ equals to the span $\mathbb Ra$ of $a$, then $a$ is called a \emph{discrete element}. By \cite[Lemma 26.2]{LZ71} it follows that every discrete element is also an atom. Moreover, by \cite[Theorem 26.4]{LZ71}, in an Archimedean vector lattice  the principal band $B_a$ generated by an atom $a$ is a projection band which is equal to $\mathbb R a$. In particular, in Archimedean vector lattices, the sets of atoms and discrete elements coincide. 

Let $a$ be an atom in an Archimedean vector lattice $X$. Then the decomposition 
$$X=\mathbb Ra \oplus \{a\}^{d}$$ is an order decomposition meaning that $x=x_1+x_2\in \mathbb Ra \oplus \{a\}^{d}$ is positive if and only if $x_1$ and $x_2$ are positive in $\mathbb Ra$ and $\{a\}^d$, respectively. For every $x\in X$ there exists a unique scalar $\varphi_a(x)$ and a unique vector $y_x \perp a$ such that $x=\varphi_a(x)a + y_x$. Since $X=\mathbb Ra \oplus \{a\}^{d}$ is an order direct sum, we have that $\varphi_a\colon X\to\mathbb R$ is a lattice homomorphism. Moreover, order continuity of the band projection onto the projection band $\mathbb Ra$ yields order continuity of $\varphi_a$. 

By Zorn's lemma type argument, every  vector lattice with at least one atom admits a maximal family $\mathcal A$ of pairwise disjoint atoms. Since for any two atoms $a_1$ and $a_2\in X$ in an Archimedean vector lattice we have that either $a_1\wedge a_2=0$ or $a_1\wedge a_2$ is simultaneously a nonzero multiple of both $a_1$ and $a_2$, the span of $\mathcal A$ is independent of the choice of $\mathcal A$. 
Therefore, an Archimedean vector lattice $X$ is atomic if and only if the ideal $I_{\mathcal A}:=\Span \mathcal A$ generated by $\mathcal A$ is order dense in $X$ for any maximal family $\mathcal A$ of pairwise disjoint atoms in $X$. If $a$ and $a'$ are  atoms in $X$ such that $a'=\lambda a$ for some $\lambda>0$, then it is easy to see that $\varphi_{a'}=\frac{1}{\lambda}\varphi_a$. 
 
Suppose now that $X$ is an atomic Archimedean vector lattice and let $\mathcal A$ be a maximal family of pairwise disjoint atoms in $X$. The following lemma characterizes positivity of a given vector $x\in X$. 

\begin{lemma}\label{pozfia je poz}
    A vector $x$ in an atomic Archimedean vector lattice is positive if and only if $\varphi_a(x)\geq 0$ for every $a\in \mathcal A$. 
\end{lemma}

In particular, a vector $x\in X$ is the zero vector if and only if $\varphi_a(x)=0$ for every $a\in \mathcal A$. The family $\mathcal F$ of all finite subsets of $\mathcal A$ is directed with set inclusion. Since  $\Span \mathcal A$ is order dense in $X$, for every positive vector $x\in X$ we have
\begin{equation}\label{zapis z atomi}
x=\sup_{a\in \mathcal A}\varphi_a(x)a=\sup_{F\in \mathcal F}\sup_{a\in F}\varphi_a(x)a=\sup_{F\in \mathcal F}\sum_{a\in F}\varphi_a(x)a. 
\end{equation}
The remaining unexplained facts about vector and normed lattices can be found in \cite{Aliprantis:06} and \cite{LZ71}.

Let $\mathscr A$ be a subalgebra of an algebra $\mathscr B$. An automorphism $\Phi$ of $\mathscr A$ is \emph{inner} if there exists an invertible element $T\in \mathscr A$ such that $\Phi(A)=TAT^{-1}$ for all $A\in \mathscr A$. If there exists an invertible element $T\in \mathscr B$ such that $\Phi(A)=TAT^{-1}$ for all $A\in \mathscr A$, then $\Phi$ is called \emph{spatial}. Spatial automorphism induced by $T$ is denoted by $\Phi_T$. Invertible elements $T_1$ and $T_2$ from $\mathscr B$ induce the same spatial automorphism of $\mathscr A$ if and only if  
$$T_2^{-1}T_1A=AT_2^{-1}T_1$$ for all $A\in \mathscr A$. The latter is equivalent to the fact that $T_2^{-1}T_1$ belongs to the centralizer 
$$C_{\mathscr B}(\mathscr A)=\{T\in \mathscr B:\; AT=TA \textrm{ for all } A\in \mathscr A\}$$
of $\mathscr A$ in $\mathscr B$.
If $\mathscr A=\mathscr B$, then $C_{\mathscr B}(\mathscr B)$ is called the center of $\mathscr B$ and is denoted by $Z(\mathscr B)$.
 It is well known that in the case  $\mathscr A=\mathscr B=\mathscr L(X)$ operators $T_1$ and $T_2$ induce the same inner automorphism if and only if $T_2^{-1}T_1$ is a nonzero scalar multiple of the identity operator.

\section{Order closedness of finite-dimensional vector subspaces}\label{Section: order closedness}

In this section we prove (see \Cref{n-dimensional order closed}) that every finite-dimensional vector subspace of an Archi\-me\-de\-an vector lattice is order closed. 
The value of this result is twofold. First, we will use it in \Cref{automorphisms of algebras}, where we study positive automorphisms of certain algebras of order continuous operators on atomic Archimedean vector lattices. Second, it is valuable in its own right, since it contributes to the broader development and understanding of concepts related to order convergence. 

For the proof of \Cref{n-dimensional order closed} we need the following lemma whose proof is provided  for reader's sake. 

\begin{lemma}\label{delta_basis}
Let $\Omega$ be a nonempty set and let $V$ be a finite-dimensional subspace of $\mathbb R^\Omega$. Then there exists a basis $f_1,\ldots,f_n$ for $V$ and points $a_1,\ldots, a_n$ in $\Omega$ such that $f_i(a_j)=\delta_{ij}$.  
\end{lemma}

\begin{proof}
We will prove the statement by induction on the dimension of $V$. The statement is clear if $\dim V=1$. 
Assume now that the statement holds for every $n$-dimensional vector subspace of $\mathbb R^\Omega$ and let $V$ be an $n+1$-dimensional subspace of $\mathbb R^\Omega$ with basis $\{h_1,h_2,\ldots, h_{n+1}\}$. 

Pick $a_{1}\in \Omega$ such that $h_{1}(a_{1})\neq 0$  
and for $2\leq i\leq n+1$ define  
$$g_i = h_i - \frac{h_i(a_1)}{h_{1}(a_{1})}h_{1}.$$
Then $g_i(a_1)=0$ for all $i>1$ and the set $\{h_1,g_2,\ldots,g_{n+1}\}$ is a basis for $V$. 
By the induction hypothesis, one can find $a_2,\ldots,a_{n+1}$ and $f_2,\ldots,f_{n+1}$ such that 
$f_i(a_j)=\delta_{ij}$ for $2\leq i,j \leq n+1$ and $\Span\{f_2,\ldots,f_{n+1}\}=\Span\{g_2,\ldots,g_{n+1}\}$.
Therefore, $f_i(a_1)=0$ for all $i>1$.
To finish the proof, we first define
$$g_1 = h_1 - \sum_{i=2}^{n+1} h_1(a_i)f_i$$ and then 
$f_1=\frac{1}{g_1(a_1)}g_1$.
\end{proof}
 
To prove \Cref{n-dimensional order closed} we need some more preparation. Let $X$ be a vector lattice. Recall that a net $(x_\alpha)_\alpha$ in $X$ \emph{converges in order} to a vector $x$ whenever there exists another net $(y_\beta)_\beta$ with $y_\beta\searrow 0$ such that for each $\beta$ there exists $\alpha_0$ such that for all $\alpha\geq \alpha_0$ we have 
$$|x_\alpha-x| \leq y_\beta.$$
Following \cite{GL18}, a set $A\subseteq X$ is \emph{order closed} in $X$ whenever it consists precisely of those elements $x\in X$ for which there exists a net $(x_\alpha)_\alpha$ in $A$ that converges in order to $x$ in $X$. Suppose additionally that $X$ is Archimedean. Then, for each positive vector $x\in X$, the principal ideal $I_x$ generated by $x$ admits a lattice norm $\|\cdot\|_x$ defined by
$$\|y\|_x:=\inf\{\lambda\geq 0:\; |y|\leq \lambda x\}.$$ If $(I_x, \|\cdot\|_x)$ is norm complete, then it is an AM-space with a unit $x$.  By the Kakutani representation theorem (see e.g. \cite[Theorem 4.29]{Aliprantis:06}), $I_x$ is lattice isometric to some $C(K)$ for some compact Hausdorff space $K$ with $x$ being mapped to the constant one function. If $(I_x,\|\cdot\|_x)$ is not norm complete, then its norm completion is an AM-space with a unit $x$, so that $I_x$ is a norm dense sublattice of some $C(K)$ space.  

\begin{theorem}\label{n-dimensional order closed}
Every finite dimensional vector subspace of an Archimedean vector lattice is order closed.
\end{theorem}

\begin{proof}
Let $V$ be a finite dimensional vector subspace of an Archimedean vector lattice $X$ and let $(v_\alpha)_\alpha$
be a net in $V$ which converges in order to some $v$ in $X$. By passing to a tail, if necessary, we may assume that the net $(v_\alpha)_{\alpha}$ is order bounded. Therefore, 
there exists a positive vector $y$ in $X$ such that $|v_\alpha|\leq y$ for every $\alpha$. 
We define 
$$x=y+\sum_{i=1}^n |e_i|$$
where $\{e_1, \ldots, e_n\}$ is some basis of $V$. 
Clearly, $V$ is contained in the principal ideal $I_x$.  
By the Kakutani representation theorem there exists a compact Hausdorff space $K$ such that $(I_x, \|\cdot\|_x)$ embeds into $C(K)$ as a norm dense sublattice with $x$ being mapped to the constant one function. 
Hence, without loss of generality we may assume that $I_x$ is contained in $C(K)$. 

By \Cref{delta_basis}, there exists a basis $\{f_1,\ldots,f_n\}$ for $V$ and points $a_1,\ldots, a_n$ in $K$ such that $f_i(a_j)=\delta_{ij}$. Hence, for each $\alpha$ we can write 
$$v_\alpha = \lambda_\alpha^{(1)} f_1+\cdots \lambda_\alpha^{(n)}f_n$$
for suitable scalars $\lambda_\alpha^{(1)}, \ldots, \lambda_\alpha^{(n)}$. 
Since the net $(v_\alpha)_\alpha$ is order bounded in $C(K)$ and for each $1\leq i\leq n$ the mapping $\varphi_i\colon C(K)\to \mathbb R$ defined by $\varphi(f)=f(a_i)$ is a lattice homomorphism, the net $(\lambda_\alpha^{(i)})_\alpha$ is bounded in $\mathbb R$ for each $1\leq i\leq n$. By passing to appropriate subnets, we may assume that for each $1\leq i\leq n$ the net $(\lambda_\alpha^{(i)})_\alpha$ already converges to some $\lambda^{(i)}\in \mathbb R$. This implies that the net $(v_\alpha)_\alpha$ converges in order in $I_x\subseteq X$ to $\lambda^{(1)} f_1+\cdots+\lambda^{(n)} f_n \in I_x$.
Since order limits are unique, we have $v=\lambda^{(1)} f_1+\cdots+\lambda^{(n)} f_n\in V$.
\end{proof}
 
\section{Operators on atomic vector lattices}\label{operators_atomic_VL}

In this section, we study and characterize order continuous operators on an Archi\-me\-de\-an vector lattice $X$ belonging to the algebra $\mathscr A_0 \subseteq \mathscr L_n(X)$ defined as follows. If $X$ does not contain atoms, we set $\mathscr A_0=\{0\}$. Otherwise, choose a maximal family $\mathcal A$ of pairwise disjoint atoms in $X$ and define $\mathscr A_0$ to be the algebra generated by all operators of the form $a\otimes \varphi_b$ for all $a, b\in \mathcal A$. It is easy to see that
\begin{equation}\label{A_0 def}
    \mathscr A_0 = \Span\{a\otimes \varphi_b:\; a,b\in \mathcal A\}.
\end{equation} 
Since each coordinate functional $\varphi_b$ is order continuous, we indeed have $\mathscr A_0 \subseteq \mathscr L_n(X)$.  

In \Cref{aabb} we prove that an order continuous operator $T$ belongs to $\mathscr A_0$ if and only if its ``matrix representation" has only finitely many nonzero entries. Moreover, \Cref{urejenostna gostost} yields that every positive operator $T\in \mathscr L_n(X)$ satisfying $TA\subseteq A$ and $T|_C=0$ is a supremum of an increasing net of positive operators from $\mathscr A_0$. In particular, we prove that whenever $X$ is Dedekind complete then $\mathscr A_0$ is an order dense vector sublattice of $\mathscr L_n(X)$ if and only if $X$ is atomic. 

We start with the following simple lemma which will be needed throughout the paper. 

\begin{lemma}\label{nicelni operator atomi}
Let $T\colon X\to Y$ be a linear operator between Archimedean vector lattices and let $\mathcal A$ and $\mathcal B$ be maximal families of pairwise disjoint atoms in $X$ and $Y$, respectively.
\begin{enumerate}
\item If $X$ is atomic and $T$ is order continuous, then $T\geq 0$ if and only if $Ta\geq 0$ for every $a\in \mathcal A$.  
\item If $Y$ is atomic, then $T\geq 0$ if and only if $\varphi_b\circ T\geq 0$ for every atom $b\in \mathcal B$.
\item If $X$ and $Y$ are atomic and $T$ is order continuous, then $T\geq 0$ if and only if $\varphi_b(Ta)\geq 0$ for all $a\in \mathcal A$
 and $b\in \mathcal B$.
\end{enumerate}
\end{lemma}

\begin{proof}
Since coordinate functionals are positive, it suffices to prove the ``if" statements. 

(i) Suppose that $Ta\geq 0$ for every atom $a\in \mathcal A$. Then $Tx\geq 0$ for every $x$ that is a positive linear combination of atoms in $\mathcal A$. Since $X$ is atomic and $T$ is order continuous, we conclude that $T$ is positive. 

(ii) Assume that $\varphi_b\circ T\geq 0$ for every atom $b\in \mathcal B$ and pick any $x\in X^+$. Then $\varphi_b(Tx)\geq 0$ for every atom $b\in \mathcal B$, so that $Tx\geq 0$  by \Cref{pozfia je poz}.  

(iii) Suppose that $\varphi_b(Ta)\geq 0$ for all $a\in \mathcal A$ and $b\in \mathcal B$. By \Cref{pozfia je poz} we have $Ta\geq 0$ for every $a\in \mathcal A$. Now we apply (i). 
\end{proof}

The following proposition characterizes order continuous operators contained in $\mathscr A_0$.

\begin{proposition}\label{aabb}
The following statements are equivalent for an order continuous operator $T$ on an Archimedean vector lattice $X$.
\begin{enumerate}
  \item $T\in \mathscr A_0$.
  \item There exist only finitely many pairs $(a,b)\in \mathcal A\times \mathcal A$ such that
$$(a\otimes \varphi_a)T(b\otimes \varphi_b)\neq 0$$
and $T$ satisfies $TA\subseteq A$ and $T|_C=0$.
\item There exist only finitely many pairs $(a,b)\in \mathcal A\times \mathcal A$ such that
$\varphi_a(Tb)\neq 0$ and $T$ satisfies $TA\subseteq A$ and $T|_C=0$. 
\end{enumerate}
\end{proposition}

\begin{proof}
    (i)$\Rightarrow$(ii) Pick $T\in \mathcal A_0$. Then $T=\sum_{c,d\in F} \lambda_{cd}(c\otimes \varphi_d) \in \mathcal A_0$ for some finite set $F\subseteq \mathcal A$ and some scalars $\lambda_{cd}$. Clearly, $T(\Span \mathcal A)\subseteq \Span \mathcal A$ and $T|_C=0$. Since $T$ is order continuous, we have $TA\subseteq A$. A direct calculation shows that 
    $$(a\otimes \varphi_a)T(b\otimes \varphi_b)=\sum_{c,d\in F} \lambda_{cd} \varphi_a(c)\varphi_d(b) (a\otimes \varphi_b)=0$$ for $(a,b)\notin F\times F$. 

    (ii)$\Rightarrow$(i) Assume now that there exists a finite set $G\subseteq \mathcal A\times \mathcal A$ such that $(a\otimes \varphi_a)T(b\otimes \varphi_b)=0$ for $(a,b)\notin G$ and $T$ satisfies $TA\subseteq A$ and $T|_C=0$.. Then $G\subseteq F\times F$ for some finite subset $F\subseteq \mathcal A$. Clearly, if $(a,b)\notin F\times F$, then $(a\otimes \varphi_a)T(b\otimes \varphi_b)=0$. It follows that $\varphi_a(Tb)=0$ for $(a,b)\notin F\times F$. We define the order continuous operator
    $$S=\sum_{c,d\in F}\varphi_c(Td)(c\otimes \varphi_d)\in \mathscr A_0$$
    and calculate
    $$\varphi_a((T-S)b)=\varphi_a(Tb)-\sum_{c,d\in F}\varphi_c(Td)\varphi_d(b)\varphi_a(c)$$
    for $a,b\in \mathcal A$. If $a\notin F$, then $\varphi_a(Tb)=0$ and $\varphi_a(c)=0$ for every $c\in F$ which clearly implies $\varphi_a((T-S)b)=0$. 
     On the other hand, if $a\in F$, then $\varphi_a(c)=0$ for every $c\in \mathcal A\setminus \{a\}$ yielding
    $$\varphi_a((T-S)b)=\varphi_a(Tb)-\sum_{d\in F}\varphi_a(Td)\varphi_d(b).$$
    If $b\notin F$, then every term in the sum above is zero, so that we have $\varphi_a(Tb)=\varphi_a(Sb)=0$. On the other hand, if $b\in F$, the sum above reduces to $\varphi_a(Tb)$. 
    Therefore, both cases give $\varphi_a((T-S)b)=0$ for all $a,b\in \mathcal A$.  \Cref{nicelni operator atomi}(iii)  implies $T|_A=S|_A$.  Since $T|_C=S|_C=0$, $S$ and $T$ agree on the order dense ideal $A\oplus C$ of $X$. To conclude the proof we once more use the fact that $T$ and $S$ are order continuous. 

    (ii)$\Leftrightarrow$(iii) follows from the identity
    $(a\otimes \varphi_a)T(b\otimes \varphi_b)=\varphi_a(Tb)(a\otimes\varphi_b)$. 
\end{proof}

If a vector lattice $X$ is finite-dimensional and Archimedean, it is atomic and order isomorphic to $\mathbb R^n$ ordered coordinatewise where $n=\dim X$. Hence, if $\{a_1,\ldots,a_n\}$ is a maximal family of pairwise disjoint atoms in $X$, then every linear operator $T$ on $X$ is of the form 
$$T=\sum_{i,j=1}^n(a_i\otimes \varphi_{a_i})T(a_j\otimes \varphi_{a_j}).$$
Suppose now that $X$ is an atomic Archimedean vector lattice and let $\mathcal A$ be a maximal family of pairwise disjoint atoms in $X$. For a finite subset $F\subseteq \mathcal A$ we define the \emph{finite truncation} $T_F$ of $T$ by 
$$T_F= \sum_{a,b\in F}(a\otimes \varphi_{a})T(b\otimes \varphi_{b}).$$
If we order the family $\mathcal F$ of all finite subsets of $\mathcal A$ by set inclusion, $\mathcal F$ becomes a directed set, so that we may consider $(T_F)_{F\in \mathcal F}$ as a net.  
The following lemma shows that every positive order continuous operator $T\in \mathscr L_n(X)$ on an atomic Archimedean vector lattice can be recovered from its finite truncations. 

\begin{lemma}\label{urejenostna gostost}
Let $X$ be an atomic Archimedean vector lattice. Then for each positive operator $T\in \mathscr L_n(X)$ the net
$(T_F)_{F\in \mathcal F}$ satisfies $T_F\nearrow T$ in $\mathscr L_n(X)$.
Moreover, for each $x\geq 0$ we have 
$T_Fx\nearrow Tx$.
\end{lemma}

\begin{proof}
Pick any finite subset $F\in \mathcal F$ and observe that $0\leq \sum_{a\in F}a\otimes \varphi_a\leq I$ yields
$0\leq T_F\leq T$. 
Clearly, positivity of $T$ implies $0\leq T_{F_1}\leq T_{F_2}$ for each pair $F_1, F_2\in \mathcal F$ with $F_1\subseteq F_2$. To prove that $T$ is the supremum of $(T_F)_{F\in \mathcal F}$ in $\mathscr L_n(X)$, let $S\in \mathscr L_n(X)$ be any upper bound of $(T_F)_{F\in \mathcal F}$. Choose any $c\in \mathcal A$, and select $F\in \mathcal F$ such that $c\in F$. By assumption, we have
$S\geq T_F\geq 0$, and so
$$Sc\geq \sum_{a,b\in F}(a\otimes \varphi_a)T(b\otimes \varphi_b)c=\sum_{a,b\in F}\varphi_b(c)\varphi_a(Tb)a=\sum_{a\in F}\varphi_a(Tc)a.$$
Since
$$Tc=\sup_{F\in \mathcal F}\sum_{a\in F}\varphi_a(Tc)a$$
by \eqref{zapis z atomi}, 
we have $Sc\geq Tc$. To conclude the proof, we apply \Cref{nicelni operator atomi}(i) for the operator $S-T$. 
\end{proof}

\begin{proposition}\label{A_0 mreza}
    For an Archimedean vector lattice $X$ the following statements hold.
    \begin{enumerate}
        \item The algebra $\mathscr A_0$ is a vector lattice. For $T=\sum_{a,b\in F}\lambda_{ab}(a\otimes \varphi_b)\in \mathscr A_0$ the modulus is given by 
        $$|T|=\sum_{a,b\in F}|\lambda_{ab}|(a\otimes \varphi_b).$$
        \item If $X$ is Dedekind complete, then $\mathscr A_0$ is a vector sublattice of $\mathscr L_n(X)$.
    \end{enumerate}
\end{proposition}

\begin{proof}
We will simultaneously prove (i) and (ii).  To this end, pick any $T\in \mathscr A_0$ and any maximal set $\mathcal A$ of pairwise disjoint atoms in $X$. Then there exists a finite subset $F\subseteq \mathcal A$ such that 
$$T=\sum_{a,b\in  F} \lambda_{ab}(a\otimes \varphi_b).$$ We claim that the operator 
$$S=\sum_{a,b\in F} |\lambda_{ab}|(a\otimes \varphi_b)$$
is the modulus of $T$ in $\mathscr A_0$. By definition, $S$ belongs to $\mathscr A_0$. Clearly, for every atom $c\in X$ we have $Sc\geq Tc,-Tc$, so that \Cref{nicelni operator atomi} yields $S|_A\geq T|_A,-T|_A$. Furthermore,  $S|_C=T|_C=0$ implies 
$S|_{A\oplus C}\geq T|_{A\oplus C},-T|_{A\oplus C}$. Therefore, restrictions of order continuous operators $S-T$ and $S+T$ to the order dense ideal $A\oplus C$ are positive which implies $S-T$ and $S+T$ are positive on $X$. This yields $S\geq T,-T$. 

Let $\widetilde S\in \mathscr L_n(X)$ be any upper bound for $\{T,-T\}$ and observe that for every $c\in \mathcal A\setminus F$ we clearly have $\widetilde Sc\geq Tc=Sc=0$. On the other hand, for $c\in  F$ we have
$$\widetilde Sc\geq \sum_{a\in F}\lambda_{ac}a \qquad \textrm{and} \qquad \widetilde Sc\geq -\sum_{a\in \mathcal F}\lambda_{ac}a,$$ which implies 
$$\widetilde Sc\geq \bigg| \sum_{a\in  F}\lambda_{ac}a \bigg| = \sum_{a\in \mathcal F}|\lambda_{ac}|a=Sc.$$
By applying \Cref{nicelni operator atomi} once again we obtain $\widetilde S|_A \geq S|_A$. Since $S|_C=0$, we also have $\widetilde S|_C\geq S|_C=0$. Similarly as before, order continuity of $S$ and $\widetilde S$ and order density of $A\oplus C$ in $X$ yield $\widetilde S\geq S$. 
This finally proves that $S$ is the supremum of the set $\{T,-T\}$ in both $\mathscr A_0$ and $\mathscr L_n(X)$. To finish the proof, note that \cite[Theorem 1.56]{Aliprantis:06} yields that $\mathscr L_n(X)$ is a vector lattice whenever $X$ is Dedekind complete. 
\end{proof}
 
We conclude this section with the following operator theoretical characterization of atomic Archimedean vector lattices. 

\begin{corollary}
   A nonzero Archimedean vector lattice $X$ is atomic if and only if for every nonzero positive operator $T\in \mathscr L_n(X)$ there exists a positive nonzero operator $S\in \mathscr A_0$ such that $0<S\leq T$.  
\end{corollary} 

\begin{proof}
    If $X$ is atomic, then every positive order continuous operator $T$ on $S$ is by \Cref{urejenostna gostost} the supremum of the increasing net $(T_F)_{F\in \mathcal F}$. 
    
    To prove the converse statement,  observe first that since the identity operator on $X$ is order continuous, the algebra $\mathscr A_0$ is, by assumption, nontrivial, yielding that the vector lattice $X$ contains atoms. Assume that $X$ is not atomic.    
    Then there exists a nonzero positive vector $x$ which is disjoint with every atom in $X$.  For an atom $c$ in $X$ we consider the rank-one operator $x\otimes \varphi_c$. By \Cref{order bounded operator functional}, the operator $x\otimes \varphi_c$ is order continuous. By assumption, there exists a nonzero positive operator $S\in \mathscr A_0$ such that  $0<S\leq x\otimes \varphi_c$. There exists a finite subset $F\subseteq \mathcal A$ and scalars $\lambda_{ab}$ such that 
    $$S=\sum_{a,b\in F}\lambda_{ab}(a\otimes \varphi_b).$$  Pick any positive vector  $y\in X$ such that $Sy\neq 0$. Then 
    $$0< \sum_{a,b\in F}\lambda_{ab}\varphi_b(y)a\leq \varphi_c(y)x$$ shows that $x$ is not disjoint with some atom which proves that the span of $\mathcal A$ is order dense in $X$.  
\end{proof}

In particular, if $X$ is Dedekind complete, then $X$ is atomic if and only if $\mathscr A_0$ is an order dense sublattice of the vector lattice $\mathscr L_n(X)$.

\section{Automorphisms of algebras of operators on atomic vector lattices}\label{automorphisms of algebras}
 
Throughout this section, we assume that $\mathcal A$ is a maximal family of pairwise disjoint atoms in an atomic Archimedean vector lattice $X$. Let $\mathscr A$ be an algebra in $\mathscr L(X)$ containing $\mathscr A_0$ (see \Cref{operators_atomic_VL}) and let $\Phi\colon \mathscr A\to \mathscr A$ be a positive algebra automorphism.  In \Cref{main theorem}, we prove that every such positive automorphism $\Phi$ satisfies $\Phi(\mathscr A_0)=\mathscr A_0$ whenever $\mathscr A\subseteq \mathscr L_n(X)$. Furthermore, if $\mathscr A\subseteq \mathscr L_n(X)$ and $X$ is Dedekind complete, in \Cref{reprezentacija avtomorfizma} we conclude that for each positive operator $T$ the restriction of $\Phi(T)$ to the ideal  generated by atoms behaves as a ``generalized permutation" operator. This result will be extremely useful in \Cref{section c00} where we consider algebras of operators on vector lattices of the form $c_{00}(\Lambda)$ for some nonempty set $\Lambda$. In that case, we prove that every positive automorphism $\Phi\colon \mathscr A\to \mathscr A$ is of the form 
$\Phi(T)=PDTD^{-1}P^{-1}$ for some ``diagonal" operator $D$ and some ``permutation" operator $P$.

We start with the following characterization of rank-one operators. 

\begin{lemma}\label{dim=1 rkT=1}
Let $T$ be an order continuous operator on an atomic Archimedean vector lattice $X$ and let $\mathscr A$ be any algebra containing $\mathscr A_0$.
Then $T$ is a rank-one operator if and only if $\dim(T\mathscr A T)=1$. 
\end{lemma}

\begin{proof}
Suppose first that the rank of $T$ is one. By \Cref{order bounded operator functional} we have $T=u\otimes \varphi$ for some nonzero vector $u\in X$ and nonzero functional $\varphi\in X_n^\sim$. Then for any $A\in \mathscr A$ we have
$$TAT=(u\otimes \varphi)A(u\otimes \varphi)=\varphi(Au)(u\otimes \varphi),$$
which shows $\dim(T\mathscr A T)\leq 1$.  Since $u\neq 0$, there exists an atom $b\in \mathcal A$ such that $\varphi_b(u)\neq 0$. Furthermore, $X$ being atomic and $\varphi$ nonzero and order continuous yield the existence of an atom $a\in \mathcal A$ such that $\varphi(a)\neq 0$. Therefore, 
$T(a\otimes \varphi_b)T=\varphi(a)\varphi_b(u)(u\otimes\varphi)\neq 0$ which proves that $\dim(T\mathscr A T)=1$.

For the opposite implication, assume that the rank of $T$ is at least two. We claim that there exist atoms $a, b\in\mathcal A$ such that $Ta$ and $Tb$ are linearly independent. If this were not the case, there would exist a nonzero vector $v\in X$ such that for each $a\in \mathcal A$ we have $Ta=\lambda_a v$ for some scalar $\lambda_a$. If $v=0$, then \Cref{nicelni operator atomi} yields $T=0$, which is impossible.

By linearity of $T$ it follows that for every $x\in \Span \mathcal A$ we have $Tx =\lambda_x v$ for some scalar $\lambda_x$. Since $X=X^+-X^+$ it suffices to prove that for every $x\in X^+$ we have $Tx=\lambda_xv$ for some scalar $\lambda_x$. To this end, since $\Span\mathcal A$ is order dense in $X$, there exists an increasing net $(x_\alpha)_\alpha$ in $\Span \mathcal A$ of positive vectors such that $0\leq x_\alpha\nearrow x$. Order continuity of $T$ yields $Tx_\alpha\goeso Tx$ and so $\lambda_{x_\alpha}v \goeso Tx$. Since the one-dimensional vector space $\mathbb Rv$ is order closed by \Cref{n-dimensional order closed}, we have $Tx=\lambda v$ for some scalar $\lambda$. This shows that $T$ is a rank-one operator which, by assumption, is impossible.  Hence,
there exist atoms $a,b\in \mathcal A$ such that $Ta$ and $Tb$ are linearly independent. 

For an arbitrary atom $e\in \mathcal A$ we define the rank one operators $X_1:=a\otimes \varphi_e$ and $X_2=b\otimes \varphi_e$. We claim that $TX_1T=Ta\otimes (\varphi_e\circ T)$ and $TX_2T=Tb\otimes (\varphi_e\circ T)$ are linearly independent for some $e\in \mathcal A$. To see this, suppose that
$$0 = \alpha (Ta\otimes (\varphi_e\circ T)) + \beta (Tb\otimes (\varphi_e\circ T)) = (\alpha Ta+\beta Tb)\otimes (\varphi_e\circ T) $$
for some scalars $\alpha$ and $\beta$. If $\alpha\neq 0$ or $\beta\neq 0$, then $\varphi_e\circ T=0$ since $Ta$ and $Tb$ are linearly independent.
By \Cref{nicelni operator atomi}, we have $T=0$, which is impossible. Hence, $TX_1T$ and $TX_2T$ are linearly independent elements in $T\mathscr AT$.
\end{proof}

The following example shows that the algebra $\mathscr A_0$ can be a proper subalgebra of $\mathscr L_n(X)$. 

\begin{example}
    Consider the vector lattice $c_{00}$ of all eventually null sequences. On $c_{00}$ we define the functional $\varphi$ by $\varphi(x)=\sum_{n=1}^\infty x_n$. Then $\varphi$ is a well-defined positive functional on $c_{00}$. Choose any net $(x_\alpha)_\alpha$ in $c_{00}$ such that $x_\alpha\searrow 0$ and pick an index $\alpha_0$. Then for all $\alpha\geq \alpha_0$ we have $0\leq x_\alpha \leq x_{\alpha_0}$ so that the net $(x_\alpha)_{\alpha\geq \alpha_0}$ is supported only on finitely many coordinates.  Since order convergence on $c_{00}$ is coordinatewise, $\varphi$ is order continuous. Clearly, for any $u\neq 0$ the operator $u\otimes \varphi$ is not contained in $\mathscr A_0$. 
\end{example}

In \Cref{c00 - operator theoretical} which is needed for the description of positive algebra automorphisms of $\mathscr L(c_{00}(\Lambda))$ we are going to prove that every linear functional on $c_{00}$ is order continuous. See also the subsequent corollary. 

Let $\mathscr A$ be a subalgebra in $\mathscr L(X)$ which contains $\mathscr A_0$. For a given positive algebra automorphism $\Phi\colon \mathscr A\to \mathscr A$ and $a,b\in \mathcal A$ we define $F_{ab}:=\Phi(a\otimes \varphi_b)$.

\begin{lemma}\label{lema o delti}
If $\Phi(\mathscr A_0)\subseteq \mathscr L_n(X)$, then the following assertions hold.
\begin{enumerate}
\item For all $a,b\in \mathcal A$ the operator $F_{ab}$ is a rank-one operator.
\item For all $a,b\in \mathcal A$ there exist positive scalars $\gamma_a$, a positive nonzero vector $u_a\in X$ and a positive nonzero order continuous functional $\psi_b$ such that
    $$F_{ab}=\tfrac{\gamma_a}{\gamma_b} (u_a\otimes \psi_b)$$
    with $\psi_a(u_a)=1$ and $\psi_a(u_b)=0$ for $b\neq a$.
\item For distinct atoms $a,b\in \mathcal A$ and an atom $e$ we have $\psi_b(e)=0$ whenever $u_a\wedge e >0$.
\end{enumerate}
\end{lemma}

\begin{proof}
(i) By \Cref{dim=1 rkT=1}, for $a,b\in \mathcal A$ the dimension of the vector space $(a\otimes \varphi_b)\mathscr A(a\otimes \varphi_b)$ is one. Since $\Phi$ is an automorphism of $\mathscr A$, the dimension of
$$\Phi\big((a\otimes \varphi_b)\mathscr A(a\otimes \varphi_b)\big)=F_{ab}\mathscr AF_{ab}$$
is also one. By \Cref{dim=1 rkT=1} and order continuity of $F_{ab}$ we conlude that $F_{ab}$ is also a rank-one operator.

(ii) By (i), the operator $F_{ab}$ is a rank-one operator. Therefore, $F_{ab}=x_{ab}\otimes f_{ab}$ for some nonzero vector $x_{ab}\in X$ and nonzero functional $f_{ab}$. Since $F_{ab}$ is order continuous, $f_{ab}$ is order continuous by  \Cref{order bounded operator functional}. 
Moreover, $f_{ab}$ is order bounded by \cite[Theorem 2.1]{AS05}. Positivity of $\Phi$ and \cite[Theorem 1.72]{Aliprantis:06} yield
$$F_{ab}=|F_{ab}|=|x_{ab}\otimes f_{ab}|=|x_{ab}|\otimes |f_{ab}|.$$ Therefore, we may assume that $x_{ab}$ and $f_{ab}$ are positive.
For an atom $a\in \mathcal A$ we define $u_a:=x_{aa}$ and $\psi_a=f_{aa}$.
Now we chose arbitrary atoms $a$ and $b$ from $\mathcal A$.
Since $(a\otimes \varphi_b)(b\otimes \varphi_a)=a\otimes \varphi_a$, we have
$$F_{ab}F_{ba}=\Phi(a\otimes \varphi_b)\Phi(b\otimes \varphi_a)=\Phi\big((a\otimes \varphi_b)(b\otimes \varphi_a)\big))=\Phi(a\otimes \varphi_a)=F_{aa}.$$
In particular, we get $$\psi_a(u_a)(u_a\otimes \psi_a)=F_{aa}^2=F_{aa}=u_a\otimes \psi_a,$$ and therefore, $\psi_a(u_a)=1$. Moreover, $$u_a\otimes\psi_a=F_{aa}=F_{ab}F_{ba}=f_{ab}(x_{ba})(x_{ab}\otimes f_{ba})$$ implies that $x_{ab}=\lambda_{ba}u_a$ and $f_{ba}=\mu_{ba}\psi_a$ and so $$F_{ab}=\gamma_{ab}(u_a\otimes\psi_b)$$ for a suitable positive scalar $\gamma_{ab}$. Fix $c\in\mathcal A$. Since $\psi_b(u_b)=1$ for $b\in\mathcal A$ we get $$F_{cc}=F_{cb}F_{bc}=\gamma_{cb}\gamma_{bc}(u_c\otimes\psi_b)(u_b\otimes\psi_c)=\gamma_{cb}\gamma_{bc}F_{cc}$$ yielding $\gamma_{cb}=\frac{1}{\gamma_{bc}}$. 
For an arbitrary $a \in \mathcal{A}$, the identity $F_{ab} = F_{ac} F_{cb}$ yields $\gamma_{ab} = \gamma_{ac} \gamma_{cb} = \frac{\gamma_{ac}}{\gamma_{bc}}$. Thus, if we define $\gamma_a:=\gamma_{ac}$ for all $a \in \mathcal{A}$, we obtain the desired expression
$$
F_{ab} = \frac{\gamma_a}{\gamma_b} (u_a \otimes \psi_b).
$$
Finally, the fact that for distinct atoms $a,b\in\mathcal A$ we have $\frac{\gamma_{b}^2}{\gamma_{a}^2}\psi_a(u_b)(u_b\otimes\psi_a)=F_{ba}^2=0$ gives us $\psi_a(u_b)=0$. 

(iii)  Choose any atom $e\in \mathcal A$ with $u_a\wedge e>0$. Since $e$ is an atom, there exists $\lambda>0$ such that $u_a\wedge e=\lambda e$.
Positivity of $\psi_b$ and the inequality
$$0=\psi_b(u_a)\geq \psi_b(u_a\wedge e)=\lambda \psi_b(e)$$ yield $\psi_b(e)=0$.
\end{proof}

For the proof of \Cref{a0 pride iz a0} we need to introduce the sets $\mathcal U_a$ and $\Psi_a$ as follows. 
Pick $a\in \mathcal A$,  and let $u_a$ and $\psi_a$ be as in \Cref{lema o delti}. We define
$${\mathcal U}_a=\{e\in \mathcal A:\; u_a\wedge e\ne 0\}=\{e\in \mathcal A:\; \varphi_e(u_a)\ne 0\}$$
and
$${\Psi}_a=\{e\in \mathcal A:\; \psi_a(e)\ne 0\}.$$

\begin{lemma}\label{o ua in psia}
For distinct atoms $a$ and $b$ in $\mathcal A$ we have 
$$\mathcal U_a \cap \mathcal U_b=\emptyset \qquad \textrm{and}\qquad \Psi_a\cap \Psi_b=\emptyset.$$
\end{lemma}

\begin{proof}
We first prove $\mathcal U_a\cap \mathcal U_b=\emptyset$ for distinct atoms $a$ and $b$ in $\mathcal A$. 
Suppose there exists an atom $e\in \mathcal U_a\cap \mathcal U_b$. 
Since $e$ is an atom, there exist positive scalars $\lambda$ and $\mu$ such that $u_a\wedge e=\lambda e$ and $u_b\wedge e=\mu e$.
Pick any atom $c$ in $X$. Then $c\neq a$ or $c\neq b$, so that $0\leq \lambda \psi_c(e)=\psi_c(u_a\wedge e)\leq \psi_c(u_a)=0$ or $0 \leq \mu \psi_c(e)=\psi_c(u_b\wedge e)\leq \psi_c(u_b)=0$ yields $\psi_c(e)=0$. In particular, we have $F_{dc}(e)=0$ for arbitrary atoms $c,d\in \mathcal A$.

Now we define $S=e\otimes \varphi_e$. Then $F_{dc}S=\frac{\gamma_d}{\gamma_c}(u_d\otimes \psi_c)(e\otimes \varphi_e)=0$. If we write $T:=\Phi^{-1}(S)$, multiplicativity of $\Phi$ gives 
$$0=\phi^{-1}(F_{dc}S)=(d\otimes \varphi_c) T=d\otimes (\varphi_c\circ T)=0$$ 
for all atoms $c, d\in \mathcal A$. Hence, $\varphi_c\circ T=0$ for every atom $c\in \mathcal A$. By \Cref{nicelni operator atomi}(ii) we conclude $T=0$ which is impossible, so $\mathcal U_a\cap \mathcal U_b=\emptyset$.

Now we prove $\Psi_a\cap \Psi_b=\emptyset$ for distinct atoms $a,b\in \mathcal A$. Suppose there exists an atom $e\in \Psi_a\cap \Psi_b$. 
Then $\psi_a(e)\ne 0$ and $\psi_b(e)\ne 0$. Since $\psi_a$ and $\varphi_e$ are positive linear functionals, for an atom
$c\in\mathcal A\setminus\{e\}$ we have $0\leq (\psi_a\wedge\varphi_e)(c)\le \varphi_e(c)=0$ and so $(\psi_a\wedge\varphi_e)(c)=0$. Moreover, by the Riesz-Kantorovich formula (see e.g. \cite[Theorem 1.18]{Aliprantis:06}) we have
\begin{align*}
 \lambda := (\psi_a\wedge\varphi_e)(e)&=\inf\{\psi_a(te)+\varphi_e((1-t)e):\; t\in[0,1]\}\\
 &=\inf\{t\psi_a(e)+(1-t):\; t\in[0,1]\}\\
 &=\min\{1, \psi_a(e)\}>0.   
\end{align*}
Therefore $\psi_a\wedge\varphi_e=\lambda\varphi_e$ as $\psi_a\wedge\varphi_e$ is order continuous and $X$ is atomic. Similarly we get $\psi_b\wedge\varphi_e=\mu\varphi_e$ for some $\mu>0$.

For every atom $c\in \mathcal A$ we have $\psi_a(u_c)=0$ or $\psi_b(u_c)=0$ since $c\ne a$ or $c\ne b$. It follows that $0\le\lambda\varphi_e(u_c)=(\psi_a\wedge\varphi_e)(u_c)\le \psi_a(u_c)=0$ or $0\le\mu\varphi_e(u_c)=(\psi_b\wedge\varphi_e)(u_c)\le \psi_b(u_c)=0$ and so
$$\varphi_e(u_c)=0$$
for every atom $c\in \mathcal A$. The operator $S=e\otimes \varphi_e$ clearly satisfies  
$$SF_{cd}=(e\otimes \varphi_e)(u_c\otimes \psi_d)=\varphi_e(u_c)(e\otimes \psi_d)=0$$
for all atoms $c,d\in \mathcal A$. It follows that
$$\Phi^{-1}(S)(c\otimes \varphi_d)=0$$
for all atoms $c,d\in \mathcal A$. By \Cref{nicelni operator atomi}(i) we conclude $\Phi^{-1}(S)=0$ which is impossible. This finally proves the claim.
\end{proof}

\begin{proposition}\label{a0 pride iz a0}
Let $\mathscr A\subseteq \mathscr L_n(X)$ be any algebra of operators which contains $\mathscr A_0$. If $\Phi$ is a positive automorphism of $\mathscr A$, then for arbitrary $T\in \mathscr A$ we have $T\in \mathscr A_0$ whenever $\Phi(T)\in \mathscr A_0$.
\end{proposition}

\begin{proof}
Suppose that for $T\in\mathscr A$ we have $S:=\Phi(T)\in\mathscr A_0$ whereas $T \notin\mathscr A_0$. By \Cref{aabb} there exist infinitely many pairs $(a,b)\in\mathcal A\times\mathcal A$ such that the product $(a\otimes\varphi_a)T(b\otimes\varphi_b)\ne 0$. Applying $\Phi$ we equivalently obtain 
$F_{aa}SF_{bb}\ne 0$ for infinitely many pairs $(a,b)\in \mathcal A\times \mathcal A$. Since $S\in\mathscr A_0$ we can write
$$S=\sum_{c,d\in F}s_{cd}(c\otimes\varphi_d)$$
for some finite set $F\subseteq \mathcal A$. To conclude the proof, we consider two cases.

\emph{Case 1}: Assume there exist infinitely many $a\in\mathcal A$ such that for some $b\in \mathcal A$ the product $F_{aa}SF_{bb}$ is nonzero. Since $F$ is finite, using \Cref{o ua in psia} we derive that $F$ intersects $\Psi_a$ for only finitely many $a\in \mathcal A$. Therefore, there exists $a\in\mathcal A$ such that $F_{aa}SF_{bb}\ne 0$ for some $b$ and $\Psi_a\cap F=\emptyset$. Consequently, since $c\notin \Psi_a$ yields $\psi_a(c)=0$, we have
$$F_{aa}S=(u_a\otimes\psi_a)\sum_{c,d\in F}s_{cd}(c\otimes\varphi_d)=\sum_{c,d\in F}s_{cd}\psi_a(c)(u_a\otimes\varphi_d)=0$$
which contradicts $F_{aa}SF_{bb}\neq 0$.

\emph{Case 2}: There exists only finitely many $a\in \mathcal A$ such that for some $b\in \mathcal A$ the product $F_{aa}SF_{bb}$ is nonzero and for some $a\in \mathcal A$ the product $F_{aa}SF_{bb}$ is nonzero for infinitely many $b\in \mathcal A$. A similar argument as in Case 1 shows that there exists $b\in\mathcal A$ such that $F_{aa}SF_{bb}\ne 0$ for some $a$ and $\mathcal U_b\cap F=\emptyset$. Since $d\notin \mathcal U_b$ yields $\varphi_d(u_b)=0$, we have
$$SF_{bb}=\sum_{c,d\in F}s_{cd}(c\otimes\varphi_d)(u_b\otimes\psi_b)=\sum_{c,d\in F}s_{cd}\varphi_d(u_b)(c\otimes\psi_b)=0$$
which is again in contradiction with  $F_{aa}SF_{bb}\ne 0$.
\end{proof}

\begin{theorem}\label{main theorem}
Let $\mathscr A\subseteq \mathscr L_n(X)$ be any algebra of operators that contains $\mathscr A_0$. The following assertions hold for a positive algebra automorphism $\Phi\colon \mathscr A\to \mathscr A$.
\begin{enumerate}
  \item There exist a bijection $\pi\colon\mathcal A\to\mathcal A$ and a set of positive scalars $\{\delta_a:\; a\in\mathcal A\}$ such that
for all $a,b\in\mathcal A$ we have
\begin{align}
\Phi(a\otimes\varphi_b)&=\tfrac{\delta_a}{\delta_b} \left(\pi(a)\otimes\varphi_{\pi(b)}\right) \label{formula1}\\
\Phi^{-1}(a\otimes\varphi_b)&=\tfrac{\delta_{\pi^{-1}(b)}}{\delta_{\pi^{-1}(a)}} \left(\pi^{-1}(a)\otimes\varphi_{\pi^{-1}(b)}\right). \label{formula2}
\end{align}
In particular,
$\Phi(\mathscr A_0)=\mathscr A_0$ and $\Phi(\mathscr A_0^+)=\mathscr A_0^+$. 
  \item The inverse $\Phi^{-1}$ is positive.
\end{enumerate}
\end{theorem}

\begin{proof}
(i) First we prove that for arbitrary $a\in\mathcal A$ the set ${\mathcal U}_a$ is a singleton set.
Assume on the contrary that for some $a$ the set $\mathcal U_a$ contains two different atoms $e$ and $f$, and pick $b, c\in \mathcal A$. Then 
$$F_{bc}e=\frac{\gamma_b}{\gamma_c}(u_b\otimes\psi_c)(e)=\frac{\gamma_b}{\gamma_c}\psi_c(e)u_b.$$
If $c\ne a$, by \Cref{lema o delti}(iii) we conclude $\psi_c(e)=0$ which implies
$F_{bc}e=0$, and similarly, $F_{bc}(f)=0$.

By \Cref{a0 pride iz a0} there exists a finite set $F\subseteq \mathcal A$ which without loss of generality contains $a$ such that $\Phi(T)=e\otimes \varphi_e +f\otimes\varphi_f$  where $T=\sum_{b,c\in F}t_{bc}(b\otimes\varphi_c)\in \mathscr A_0$. Then 
$$\Phi(T)=\sum_{b,c\in F}t_{bc}F_{bc}.$$
Since for $c\ne a$ we have $F_{bc}e=0$ and $F_{ba}e=\frac{\gamma_b}{\gamma_a}\psi_a(e)u_b$ it follows that
$$e=\Phi(T)e=\sum_{b,c\in F}t_{bc}F_{bc}e=\sum_{b\in F}t_{ba}F_{ba}e= \tfrac{\gamma_b}{\gamma_a}\psi_a(e)\sum_{b\in\mathcal F}t_{ba}u_b.$$
Similarly we get
$$f=\tfrac{\gamma_b}{\gamma_a}\psi_a(f)\sum_{b\in F}t_{ba}u_b.$$
which is a contradiction since $e$ and $f$ are linearly independent. This proves that for each $a\in \mathcal A$ the set $\mathcal U_a$ is a singleton set which yields an injective mapping $\pi\colon \mathcal A\to \mathcal A$ such that ${\mathcal U}_a=\{\pi(a)\}$. 

Since $\mathcal U_a$ consists of all atoms that are not disjoint with $u_a$ and $X$ is atomic, an application of \eqref{zapis z atomi} yields $u_a=\kappa_a\pi(a)$ for some positive scalar $\kappa_a$. Hence, 
$$\Phi(a\otimes \varphi_b)=\kappa_a\pi(a)\otimes \psi_b.$$
To prove that $\pi$ is surjective, choose an atom $c\in\mathcal A$. Since $\Phi$ is an automorphism of $\mathscr A$, there exists $T\in \mathscr A$ such that $\Phi(T)=c\otimes \varphi_c$. By
\Cref{a0 pride iz a0} we conclude $T\in \mathscr A_0$. Therefore, there exists a finite set $F\subseteq \mathcal A$ such that 
$T=\sum_{a,b\in F} \lambda_{ab}(a\otimes \varphi_b)$ for some scalars $\lambda_{ab}$.  Since
$$c=(c\otimes \varphi_c)c=\Phi(T)c=\sum_{a,b\in F}\kappa_a (\pi(a)\otimes \psi_b)c=\sum_{a,b\in F}\kappa_a \psi_b(c)\pi(a),$$ and $\pi(a)\in \mathcal A$ for each $a\in F$, we conclude that $c=\pi(a)$ for some $a\in \mathcal A$. 

For different atoms $a,b\in\mathcal A$ we have $\kappa_a\psi_b(\pi(a))=\psi_b(u_a)=0$. Since $\kappa_a\ne 0$ we get $\psi_b(\pi(a))=0$ for all $a\ne b$. Order continuity of $\psi_b$ implies that
$$\psi_b=\eta_b\varphi_{\pi(b)}$$
for some $\eta_b>0$, and so, finally
$\Phi(a\otimes\varphi_b)=\kappa_a\eta_b(\pi(a)\otimes\varphi_{\pi(b)}).$
Since $(\Phi(a\otimes\varphi_a))^2=\Phi(a\otimes\varphi_a)$ it follows that $\kappa_a\eta_a=1$, yielding $\eta_a=\frac{1}{\kappa_a}$.
By \Cref{lema o delti} we conclude
$$\Phi(a\otimes\varphi_b)=\tfrac{\gamma_a}{\gamma_b} (u_a\otimes \psi_b)=\tfrac{\gamma_a\kappa_a}{\gamma_b\kappa_b}(\pi(a)\otimes\varphi_{\pi(b)}).$$
By introducing $\delta_a=\gamma_a\kappa_a$ for each $a\in \mathcal A$ we obtain \eqref{formula1}. One can verify \eqref{formula2} by a direct calculation. 

To prove $\Phi(\mathscr A_0)=\mathscr A_0$ and $\Phi(\mathscr A_0^+)=\mathscr A_0^+$, note that \Cref{A_0 mreza}(i), and \eqref{formula1} and \eqref{formula2} imply
$\Phi(\mathscr A_0^+)\subseteq  \mathscr A_0^+$ and $\Phi^{-1}(\mathscr A_0^+)\subseteq \mathscr A_0^+$, respectively, yielding $\Phi(\mathscr A_0^+)=\mathscr A_0^+$. Since $\mathscr A_0=\mathscr A_0^+-\mathscr A_0^+$, we obtain $\Phi(\mathscr A_0)=\mathscr A_0$.

(ii) Pick a positive operator $T$ in $\mathscr A$ and atoms $a,b\in \mathcal A$. By \eqref{formula2}, there exist atoms $c,d\in \mathcal A$ such that
$\Phi^{-1}(c\otimes \varphi_c)=b\otimes \varphi_b$ and $\Phi^{-1}(d\otimes \varphi_d)=a\otimes \varphi_a$.
Since $(c\otimes \varphi_c)T(d\otimes \varphi_d)=\varphi_c(Td)(c\otimes \varphi_d) \in \mathscr A_0^+$ and by (i) we have $\Phi^{-1}(\mathscr A_0^+)=\mathscr A_0^+$, we conclude that
\begin{align*}
\Phi^{-1}((c\otimes \varphi_c)T(d\otimes \varphi_d))a&=(b\otimes \varphi_b)\Phi^{-1}(T)(a\otimes \varphi_a)a=\varphi_b(\Phi^{-1}(T)a)b
\end{align*}
is a positive vector. In particular, $\varphi_b(\Phi^{-1}(T)a)\geq 0$ for all $a,b\in \mathcal A$. To finish the proof we apply the fact that $\Phi^{-1}(T)$ is order continuous and \Cref{nicelni operator atomi}. 
\end{proof}

Let $\mathscr A\subseteq \mathscr L_n(X)$ be any algebra of operators which contains $\mathscr A_0$ and let $\Phi\colon \mathscr A\to \mathscr A$ be a positive algebra automorphism $\Phi\colon \mathscr A\to \mathscr A$. By \Cref{main theorem} there exists a
bijection $\pi\colon\mathscr A\to\mathscr A$ and a set of positive scalars $\{\delta_a:\; a\in\mathcal A\}$ such that
for all $a,b\in\mathcal A$ we have
$$\Phi(a\otimes\varphi_b)=\tfrac{\delta_a}{\delta_b} \left(\pi(a)\otimes\varphi_{\pi(b)}\right).$$
Since $\mathcal A$ is a Hamel basis for $I_{\mathcal A}$, we can define $P, D\colon I_{\mathcal A}\to I_{\mathcal A}$ given by
$Pa=\pi(a)$ and $Da=\delta_a a$ for each $a\in \mathcal A$. Note that  $P$ and $D$ are invertible  on $I_{\mathcal A}$ with their inverses given by
$P^{-1}a=\pi^{-1}(a)$ and  $D^{-1}a=\tfrac{1}{\delta_a}a$ for each $a\in \mathcal A$. We will call operators $P$ and $D$ a \emph{permutation} and a \emph{diagonal} operator, respectively. Conversely, for any set $\{\delta_a:\; a\in \mathcal A\}$ of real numbers  we can define the diagonal operator on $I_{\mathcal A}$ as above. The set $\{\delta_a:\; a\in \mathcal A\}$ is called the set of all \emph{diagonal coefficients} of the operator $D$.  The operator $D\colon I_{\mathcal A}\to I_{\mathcal A}$ is positive if and only if $\delta_a\geq 0$ for every $a\in \mathcal A$ and $D$ is bijective if and only if $\delta_a\neq 0$ for every $a\in \mathcal A$. 

The following lemma will be needed in the proof of \Cref{reprezentacija avtomorfizma}.
 
\begin{lemma}\label{pointwise_suprema}
Let  $\mathscr A\subseteq \mathscr L_n(X)$ be a subalgebra containing $\mathscr A_0$ and $\Phi\colon \mathscr A\to \mathscr A$ be a positive algebra automorphism. Then for every $T\in \mathscr A$ we have
$$\Phi(T_F)=\Phi(T)_{\pi(F)}$$
for every finite subset $F\subseteq \mathcal A$. Moreover, for every positive operator $T\in \mathscr A$ and every positive vector $x\in X$  we have 
$\Phi(T_F)x\nearrow \Phi(T)x$ where $F$ runs over the family $\mathcal F$ of all finite subsets of $\mathcal A$ ordered by set inclusion. 
\end{lemma}

\begin{proof}
Let us choose a finite subset $F\subseteq \mathcal A$ and write
$T_F=\sum_{a,b\in F}(a\otimes \varphi_a)T(b\otimes \varphi_b)$. Since $\Phi$ is an algebra homomorphism, by \Cref{main theorem} we have
\begin{align*}
\Phi(T_F) &= \sum_{a,b\in F} \Phi(a\otimes \varphi_a)\Phi(T)\Phi(b\otimes \varphi_b)
= \sum_{a,b\in F} (\pi(a)\otimes \varphi_{\pi(a)}) \Phi(T) (\pi(b)\otimes \varphi_{\pi(b)}) \\
&= \sum_{c,d \in \pi(F)}(c\otimes \varphi_c)\Phi(T)(d\otimes \varphi_d)= \Phi(T)_{\pi(F)}.
\end{align*}
To prove the moreover statement, note that \Cref{urejenostna gostost} yields 
$$\Phi(T_F)x=\Phi(T)_{\pi(F)}x\nearrow \Phi(T)x$$ since $\pi\colon \mathcal A\to \mathcal A$ is a bijection. 
\end{proof}

\begin{theorem}\label{reprezentacija avtomorfizma}
Let $\mathscr A\subseteq \mathscr L_n(X)$ be a subalgebra containing $\mathscr A_0$ and let $\Phi\colon \mathscr A\to \mathscr A$ be a positive algebra automorphism. Then for each $T\in \mathscr A_0$ the ideal $I_{\mathcal A}$ is invariant under $\Phi(T)$ and we have
$$\Phi(T)|_{_{I_{\mathcal A}}}=PDTD^{-1}P^{-1}|_{_{I_{\mathcal A}}}.$$ 
Furthermore, for each positive $T\in \mathscr A$ we have
$$\Phi(T)|_{_{I_{\mathcal A}}}=\sup \{PDSD^{-1}P^{-1}|_{_{I_{\mathcal A}}}:\; S\in [0,T]\cap \mathscr A_0\}.$$
\end{theorem}

\begin{proof}
Since the set $\{a\otimes \varphi_b:\; a,b\in \mathcal A\}$ spans $\mathscr A_0$, by linearity of $\Phi$ it suffices to prove $\Phi(a\otimes \varphi_b)|_{I_{\mathcal A}}=PD(a\otimes \varphi_b)D^{-1}P^{-1}|_{_{I_{\mathcal A}}}$.
Pick any $c\in \mathcal A$. Then
\begin{align*}
PD(a\otimes \varphi_b)D^{-1}P^{-1}c& = PD(a\otimes \varphi_b)D^{-1}\pi^{-1}(c)=\frac{1}{\delta_{\pi^{-1}(c)}}PD(a\otimes \varphi_b)\pi^{-1}(c)\\
&=\frac{\varphi_b(\pi^{-1}(c))}{\delta_{\pi^{-1}(c)}}PDa=\frac{\varphi_b(\pi^{-1}(c))}{\delta_{\pi^{-1}(c)}}\,\delta_a\,\pi(a).
\end{align*}
If $c=\pi(b)$, then $PD(a\otimes \varphi_b)D^{-1}P^{-1}c=\frac{\delta_a}{\delta_b}\pi(a)$. Otherwise, we have $PD(a\otimes \varphi_b)D^{-1}P^{-1}c=0$ proving $\Phi(T)|_{_{I_{\mathcal A}}}=PDTD^{-1}P^{-1}|_{_{I_{\mathcal A}}}$.

To prove the second formula, choose a positive operator $T\in \mathscr L_n(X)$. By  \Cref{urejenostna gostost} we have that $T=\sup_{F\in \mathcal F}T_F$ is the supremum of the increasing net $(T_F)_{F\in \mathcal F}$ where $F$ runs over the family $\mathcal F$ of all finite subsets of $\mathcal A$ ordered by set inclusion. Clearly, we have $T_F\in [0,T]\cap \mathscr A_0$, and for each 
$S\in [0,T]\cap \mathscr A_0$ we can find $F\in \mathcal F$ such that $S\leq T_F$. By \Cref{pointwise_suprema} we have $\Phi(T_F)x\nearrow \Phi(T)x$ for each positive vector $x\in X$. In particular, this holds for every positive vector $x\in I_{\mathcal A}$ which yields $\Phi(T_F)|_{I_{\mathcal A}}\nearrow \Phi(T)|_{I_{\mathcal A}}$. To finish the proof note that we have

\begin{align*}
\Phi(T)|_{I_{\mathcal A}}&=\sup_{F\in \mathcal F}\Phi(T_F)|_{I_{\mathcal A}}=\sup \{\Phi(S)|_{I_{\mathcal A}}:\; S\in [0,T]\cap \mathscr A_0\}\\
&=\sup \{PDSD^{-1}P^{-1}|_{_{I_{\mathcal A}}}:\; S\in [0,T]\cap \mathscr A_0\}.\qedhere
\end{align*}
\end{proof}

For the converse, assume $\Phi\colon \mathscr A\to \mathscr A$ is a linear operator such that there exist a positive invertible diagonal operator $D$ and a permutation operator $P$ such that for each $T\in \mathscr A_0$ the ideal $I_{\mathcal A}$ is invariant under $\Phi(T)$ and we have
$$\Phi(T)|_{_{I_{\mathcal A}}}=PDTD^{-1}P^{-1}|_{_{I_{\mathcal A}}}.$$ 
Then for $T,S\in \mathscr A$ we have 
\begin{align*}
\Phi(TS)\big|_{I_{\mathcal A}}
&= PDTSD^{-1}P^{-1}\big|_{I_{\mathcal A}}
= PDTD^{-1}P^{-1} PDSD^{-1}P^{-1}\big|_{I_{\mathcal A}} \\
&= PDTD^{-1}P^{-1}\big|_{I_{\mathcal A}} \; PDSD^{-1}P^{-1}\big|_{I_{\mathcal A}}\\
&=\Phi(T)\big|_{I_{\mathcal A}}\Phi(S)\big|_{I_{\mathcal A}}
\end{align*}
Since operators in $\mathscr A$ are order continuous and $I_{\mathcal A}$ is order dense in $X$, we have $\Phi(TS)=\Phi(T)\Phi(S)$ for all $T,S\in \mathscr A$. 

\section{Algebras of operators on $c_{00}(\Lambda)$} \label{section c00}

Let $X$ be an atomic Archimedean vector lattice. If we fix any maximal set $\mathcal A$ of pairwise disjoint atoms in $X$, then the ideal $I_{\mathcal A}$ generated by atoms in $\mathcal A$ equals the linear span of $\mathcal A$. Hence, $I_{\mathcal A}$ can be realized as the vector lattice $c_{00}(\mathcal A)$ of all finitely supported functions defined on $\mathcal A$. In order to better understand the structure of positive automorphisms from \Cref{reprezentacija avtomorfizma} we restrict ourself to the special case where the underlying vector lattice $X$ is the vector lattice
$c_{00}(\Lambda)$ equipped with the supremum norm. By $e_\lambda$ we denote the characteristic function of the set $\{\lambda\}$. When $\Lambda=\mathbb N$, we rather write $c_{00}$ instead of $c_{00}(\mathbb N)$. Since it is standard to interpret elements of $c_{00}$ as eventually null sequences, we can also treat elements of $c_{00}(\Lambda)$ as finitely supported ``sequences" indexed by the index set $\Lambda$.  As special cases of \Cref{c00 version} and \Cref{c00 bounded version}, we deduce that every positive algebra automorphism of $\mathscr{L}(c_{00}(\Lambda))$ and $\mathscr{B}(c_{00}(\Lambda))$ is inner.
Since \Cref{reprezentacija avtomorfizma} is applicable only to subalgebras contained in the algebra of order continuous operators, it is important to find all order continuous operators in $\mathscr{B}(c_{00}(\Lambda))$. 
In \Cref{c00 - operator theoretical} we provide an operator theoretical characterization of vector lattices that are lattice isomorphic to a vector lattice of the form $c_{00}$. In particular, it follows that every linear operator on $c_{00}(\Lambda)$ is order continuous, in particular, also order bounded by \cite[Theorem 2.1]{AS05}.
 
\begin{theorem}\label{c00 - operator theoretical}
For an Archimedean vector lattice $X$ the following statements are equivalent.
\begin{enumerate}
  \item $X$ is lattice isomorphic to $c_{00}(\Lambda)$ for some non-empty index set $\Lambda$.
  \item $\mathscr L(X,Y)=\mathscr L_n(X,Y)$ for every Archimedean vector lattice $Y$.
  \item $\mathscr L(X,Y)=\mathscr L_b(X,Y)$ for every Archimedean vector lattice $Y$.
  \item $X^\sim=X'$.
\end{enumerate}
\end{theorem}

\begin{proof}
(i) $\Rightarrow$ (ii)
Assume first that $X$ is lattice isomorphic to $c_{00}(\Lambda)$ for some non-empty set $\Lambda$. Since the inclusion $\mathscr L_n(X,Y)\subseteq \mathscr L(X,Y)$ always holds, it suffices to prove $\mathscr L(X,Y)\subseteq \mathscr L_n(X,Y)$.

Pick any lattice isomorphism $T\colon c_{00}(\Lambda)\to X$. If $S_0\in \mathscr L(X,Y)$, then $S_0T\in \mathscr L(c_{00}(\Lambda),Y)$. We claim that $S:=S_0T\in \mathscr L_n(c_{00}(\Lambda),Y)$. To this end, pick any net $(x_\alpha)_\alpha$ in $c_{00}(\Lambda)$ that converges in order to $0$. By passing to a tail, if necessary, we may assume that the net $(x_\alpha)_\alpha$ is order bounded. Hence, there exists $x\geq 0$ in $c_{00}(\Lambda)$ such that $0 \leq |x_\alpha|\leq x$ for every $\alpha$. There exists a finite subset $\Lambda_0\subseteq \Lambda$ such that $x=\sum_{\lambda\in\Lambda_0}x^{(\lambda)} e_\lambda$. Since $x_\alpha\to 0$ in order, it follows that for every $\lambda\in \Lambda_0$ we have  $x_\alpha^{(\lambda)}\to 0$ in $\mathbb R$. Pick any $\varepsilon>0$. Since $\Lambda_0$ is finite, there exists an index $\alpha_\varepsilon$ such that for all $\alpha\geq \alpha_\varepsilon$ and each $\lambda\in \Lambda_0$ we have $0 \leq |x_\alpha^{(\lambda)}|<\varepsilon$. Let us denote by  $f$ the vector $\sum_{\lambda\in \Lambda_0}|Se_\lambda|$. Then for every $\alpha\geq \alpha_\varepsilon$ we have
$$0\leq |Sx_\alpha|=\left|\sum_{\lambda\in \Lambda_0}x_\alpha^{(\lambda)}Se_\lambda\right|\leq \varepsilon f.$$
Since $Y$ is Archimedean, we have $\varepsilon f \searrow 0$ as $\varepsilon \searrow 0$, so that $Sx_\alpha\to 0$ in order which proves order continuity of $S$. Since $T$ is a lattice isomorphism, it follows that
$S_0=ST^{-1}$ is order continuous.

(ii) $\Rightarrow$ (iii)
Since every order continuous operator is order bounded by \cite[Theorem 2.1]{AS05}, we have
$$\mathscr L(X,Y) = \mathscr L_n(X,Y) \subseteq \mathscr L_b(X,Y)\subseteq  \mathscr L(X,Y).$$

(iii) $\Rightarrow$ (iv) We take $Y=\mathbb R$. 

(iv) $\Rightarrow$ (i)
Pick any positive vector $x\in X$ and consider the principal ideal $I_x$ equipped with the norm $\|\cdot\|_x$. We claim that $I_x$ is finite-dimensional. If this were not the case, then by \cite[Theorem 26.10]{LZ71} $I_x$ would contain an infinite sequence $(e_n)_{n\in\mathbb N}$ of pairwise disjoint positive vectors. By scaling, if necessary, assume that $0\leq e_n\leq x$. If we define the vector
 $y_n:=\sum_{k=1}^n e_k$, disjointness and positivity of vectors $e_1,\ldots,e_n$ yields
 $$0\leq y_n=\sum_{k=1}^ne_k=\bigvee_{k=1}^ne_k\leq x.$$
 Since disjoint vectors are always linearly independent, the set $\{e_n:\; n\in \mathbb N\}$ can be extended to a basis of $X$. Pick any linear functional $\varphi$ with $\varphi(e_n)=1$ for each $n\in \mathbb N$. Since, by assumption, $\varphi$ is order bounded, by the Riesz-Kantorovich theorem \cite[Theorem 1.18]{Aliprantis:06} there exists $|\varphi|\colon X\to \mathbb R$. As $|\varphi(z)|\leq |\varphi|(|z|)$ for every vector $z$, for each $n\in\NN$ we get
 $$|\varphi|(x)\geq |\varphi|(y_n)\geq \varphi(y_n)=n$$
 which is clearly impossible. This contradiction shows that $I_x$ is finite-dimensional.  By \cite[Theorem 61.4]{LZ71} it follows that $X$ is lattice isomorphic to $c_{00}(\Lambda)$ for some nonempty set $\Lambda$.
\end{proof}

\begin{corollary} 
For a nonempty set $\Lambda$ and for every Archimedean vector lattice $Y$ we have
$$\mathscr L(c_{00}(\Lambda),Y)=\mathscr L_b(c_{00}(\Lambda),Y)=\mathscr L_n(c_{00}(\Lambda),Y).$$
In particular $c_{00}(\Lambda)'=c_{00}(\Lambda)^\sim=c_{00}(\Lambda)_n^\sim$.
\end{corollary}

The following corollary implies that every positive algebra automorphism of ${\mathscr L}(c_{00}(\Lambda))$ is inner.

\begin{corollary}\label{c00 version}
Let $\Lambda$ be a nonempty set and
let ${\mathscr A}$ be subalgebra of ${\mathscr L}(c_{00}(\Lambda))$ which contains $\mathscr A_0$. If $\Phi\colon \mathscr A\to \mathscr A$ is a positive algebra automorphism, then
$$\Phi(T)=PDTD^{-1}P^{-1}$$
for some permutation operator $P$ and some positive invertible diagonal operator $D$ on $c_{00}(\Lambda)$.
\end{corollary}

\begin{proof}
Observe first that by \Cref{c00 - operator theoretical} we have $\mathscr L_n(c_{00})=\mathscr L(c_{00})$.
Since $I_{\mathcal A}=c_{00}(\Lambda)$,
by \Cref{reprezentacija avtomorfizma} we have
$$\Phi(T)=\sup \{PDSD^{-1}P^{-1}:\; S\in [0,T]\cap \mathscr A_0\}$$
for each positive operator $T$ on $c_{00}(\Lambda)$.
To prove $\Phi(T)=PDTD^{-1}P^{-1}$ we will apply \cite[Theorem 1.19]{Aliprantis:06}.
Note first, that the set $[0,T]\cap \mathscr A_0$ can be considered as an increasing net $(T_\alpha)_\alpha$ with supremum $T$.
Hence, for each positive vector $x$ we have
$0\leq T_\alpha x\nearrow Tx$. If we replace $x$ with $D^{-1}P^{-1}x$, we obtain $0\leq T_\alpha D^{-1}P^{-1}x\nearrow TD^{-1}P^{-1}x$.
Order continuity of $PD$ yields $0\leq PDT_\alpha D^{-1}P^{-1}x\nearrow PDTD^{-1}P^{-1}x$ so that $\Phi(T)=PDTD^{-1}P^{-1}$.
\end{proof}

\begin{example}
Let $\mathscr A$ be the algebra in $\mathscr L(c_{00})$ generated by the identity operator $I$ and $\{e_i\otimes \varphi_{e_j}:\; i,j\in \mathbb N\}$.
If $P$ is the permutation operator induced by the bijection $\pi$ defined as $\pi(2k-1)=2k$ and $\pi(2k)=2k-1$, then it is easy to see that $\Phi\colon T\mapsto PTP^{-1}$ is a positive automorphism of $\mathscr A$ which is not inner. 
\end{example}

\Cref{c00 version} is, in particular, applicable to subalgebras $\mathscr A$ of bounded operators on $\mathscr L(c_{00}(\Lambda))$. One could expect that $P$ and $D$ obtained by  \Cref{c00 version} are bounded. A direct verification shows that $P$ is indeed bounded, whereas to prove boundedness of $D$ we will need the following lemma which explicitly provides the operator norm of a bounded operator on the normed space $c_{00}(\Lambda)$. 

\begin{lemma}\label{Lc00 description}
A linear operator $T\colon c_{00}(\Lambda)\to c_{00}(\Lambda)$ is bounded if and only if
$$M:=\sup_{\lambda\in\Lambda} \sum_{\mu\in\Lambda} |\varphi_{e_\lambda}(Te_\mu)|<\infty.$$
Moreover, if $T$ is bounded then $\|T\|=M$.
\end{lemma}

\begin{proof}
Without loss of generality we may assume that $T$ is nonzero.
Assume first that $T$ is bounded and pick $\lambda\in \Lambda$. 
For a finite subset $F\subseteq\Lambda$ we define
$$x=\sum_{\mu\in F} \sgn (\varphi_{e_\lambda}(Te_\mu))e_\mu \in c_{00}(\Lambda)$$ and observe that $\|x\|\le 1$. Since $Tx\in c_{00}(\Lambda)$, for any finite subset $G\subseteq \Lambda$ containing the support $F_{Tx}$ of $Tx$ we have
$$Tx=\sum_{\lambda\in G} \left(\sum_{\mu\in F} \sgn (\varphi_{e_\lambda}(Te_\mu))\varphi_{e_\lambda}(Te_\mu)\right)e_\lambda.$$
Pick any $\lambda_0\in \Lambda$ and define $G:=F_{Tx}\cup\{\lambda_0\}$. As $\|x\|\le 1$, we obtain
\begin{align*}
\|T\|&\geq \|Tx\|=\max_{\lambda\in G} \left |\sum_{\mu\in F} \sgn (\varphi_{e_\lambda}(Te_\mu))\varphi_{e_\lambda}(Te_\mu)\right |= \max_{\lambda\in G} \sum_{\mu\in F} |\varphi_{e_\lambda}(Te_\mu)|\\
& \geq \sum_{\mu\in F} |\varphi_{e_{\lambda_0}}(Te_\mu)|.
\end{align*}
By definition of convergence we get $\|T\|\ge \sum_{\mu\in\Lambda} |\varphi_{e_{\lambda_0}}(Te_\mu)|$. Since $\lambda_0\in \Lambda$ was arbitrary, we conclude $\|T\|\geq M$.

To prove the converse statement, assume that $M<\infty$ and pick any unit vector $x=\sum_{\mu\in F_x} x_\mu e_\mu\in c_{00}(\Lambda)$ where $x_\mu=\varphi_{e_\mu}(x)$. Then $|x_\mu|\leq 1$ for every $\mu\in\Lambda$. Pick any $\lambda\in\Lambda$. Then the $\lambda$-th component of $Tx$ equals $\varphi_{e_\lambda}(Tx)=\sum_{\mu\in F_x}\varphi_{e_\lambda}(Te_\mu) x_\mu$ and so
$$|\varphi_{e_\lambda}(Tx)|\leq \sum_{\mu \in F_x}|\varphi_{e_\lambda}(Te_\mu)||x_\mu|\leq \sum_{\mu \in F_x}|\varphi_{e_\lambda}(Te_\mu)|\leq M.$$
Hence, $\|Tx\|\leq M$ for every unit vector $x\in c_{00}(\Lambda)$ which finally gives $\|T\|\leq M$.

For the moreover statement, observe that a combination of the conclusions above implies $\|T\|=M$.
\end{proof}

\begin{corollary}\label{c00 bounded version}
Let $\Lambda$ be a nonempty set. Then for every positive automorphism $\Phi$ of $\mathscr B(c_{00}(\Lambda))$ there exist a permutation operator $P$ and a bounded positive diagonal operator $D$ on $c_{00}(\Lambda)$ with bounded inverse such that
$$\Phi(T)=PDTD^{-1}P^{-1}$$ for each $T\in \mathscr B(c_{00}(\Lambda))$.
\end{corollary}

In particular, every positive automorphism $\Phi$ of $\mathscr B(c_{00}(\Lambda))$ is inner. 

\begin{proof}
By \Cref{c00 version} there exist a permutation operator $P$ and a positive diagonal operator $D$ such that $\Phi(T)=PDTD^{-1}P^{-1}$ for each $T\in \mathscr B(c_{00}(\Lambda))$. We need to prove that $D$ and $D^{-1}$ are bounded. It suffices to see that the sets of diagonal coefficients $\{d_\lambda:\; \lambda\in \Lambda\}$ and $\{d_\lambda^{-1}:\; \lambda\in \Lambda\}$ are bounded. To this end, we will apply \Cref{Lc00 description}. 

By assumption,  for every $T\in\mathscr B(c_{00}(\Lambda))$ we know that $\Phi(T)=PDTD^{-1}P^{-1}$ belongs to $\mathscr B(c_{00}(\Lambda))$ yielding that $DTD^{-1}$ is a bounded operator on $c_{00}(\Lambda)$.  Therefore, by \Cref{Lc00 description} we have
\begin{align*}
\|DTD^{-1}\| &= \sup_{\lambda\in\Lambda} \sum_{\mu\in\Lambda} |\varphi_{e_\lambda}(DTD^{-1}e_\mu)|.
\end{align*}
Let us fix $\alpha\in\Lambda$ and choose an absolutely convergent series $\sum_{\mu\in \Lambda}t_\mu$. 
For each $\mu\in\Lambda$ we define $Te_\mu=t_\mu e_\alpha$. By linearity, $T$ can be uniquely extended to a linear operator (again denoted by $T$) on $c_{00}(\Lambda)$. Using \Cref{Lc00 description} we see that $T$ is bounded. 

Since $DTD^{-1}$ is bounded and  $DTD^{-1}e_{\mu}=d_{\alpha}d_{\mu}^{-1} t_\mu e_\alpha$, by \Cref{Lc00 description} we conclude
$$\|DTD^{-1}\|=\sum_{\mu\in\Lambda} |{d_{\alpha}}d_{\mu}^{-1} t_\mu|=d_\alpha \sum_{\mu\in\Lambda} d_\mu^{-1}|t_{\mu}|<\infty.$$
Therefore, for every absolutely convergent series $\sum_{\mu\in \Lambda}t_\mu$ the series $\sum_{\mu\in\Lambda} d_\mu^{-1}|t_{\mu}|$ also converges.

For a finite subset $K\subseteq\Lambda$ we define a functional $\varphi_K\colon\ell^1(\Lambda)\to\RR$ as
$$\varphi_K((t_\mu)_\mu)=\sum_{\mu\in K} d_\mu^{-1}t_\mu.$$
Then $\|\varphi_K\|=\max_{\mu\in K} d_\mu^{-1}$. For a fixed $t=(t_\mu)_\mu$, the set 
$$\{\varphi_K(t):\; K\subseteq \Lambda\ {\rm finite}\}$$ is bounded since $|\varphi_K(t)|\le \sum_{\mu\in \Lambda} d_\mu^{-1}|t_\mu|$. The principle of uniform boundedness implies that the norms $\{\|\varphi_K\|:\; K\subseteq \Lambda\ {\rm finite}\}$ are bounded yielding that $(d_\mu^{-1})_{\mu\in\Lambda}$ is bounded. Therefore $D^{-1}\in\mathscr B(c_{00}(\Lambda))$.

Since, clearly the linear mapping $\Psi\colon T\mapsto P^{-1}D^{-1}TDP$ is the inverse of $\Phi$, it is a positive automorphism of $\mathscr B(c_{00}(\Lambda))$. By the proof above, $D$ is bounded which completes the proof.
\end{proof}

\section{Concluding remarks on order bounded functionals}\label{Concluding}

By \Cref{c00 - operator theoretical}, an Archimedean vector lattice $X$ admits a non-order bounded functional if and only if $X$ is not lattice isomorphic to a vector lattice of the form $c_{00}(\Lambda)$ for some set $\Lambda$. 
In non-Archimedean vector lattices, constructing non-order bounded functionals is simple, as the so-called infinitely small elements come to our help. Recall that a positive vector $x$ of a vector lattice $X$ is \emph{infinitely small} whenever there exists a positive vector $y$ such that for all $n\in \mathbb N$ we have $0\leq nx\leq y$. Clearly, a vector lattice is not Archimedean if and only if it contains nonzero infinitely small positive vectors. 

\begin{example}\label{non archimedean non order bounded}
    Let $X$ be a non-Archimedean vector lattice. Then there exist positive nonzero vectors $x,y\in X$ such that for each $n\in \mathbb N$ we have $0<nx\leq y$. Let $\varphi$ be any linear functional on $X$ such that $\varphi(x)=1$. Then for each $\alpha>0$  we have $\alpha=\varphi(\alpha x)$. Since $0\leq \alpha\leq n$ yields $0\leq \alpha x\leq nx\leq y$, we conclude   $[0,\infty)\subseteq\varphi([0,y])$, and so, $\varphi$ is not order bounded. In particular, every order bounded functional vanishes on the set of infinitely small vectors. 
\end{example}

Recall that the lexicographical product $X\circ Y$ of vector lattices $X$ and $Y$ is the vector space $X\times Y\cong X\oplus Y$ equipped with the lexicographical partial ordering $\leq_{\mathrm{Lex}}$ defined as 
$(x_1, y_1) \leq_{\mathrm{Lex}} (x_2, y_2)$ if $x_1 < x_2$ or $x_1=x_2$ and $y_1\leq y_2$. It is a standard exercise to prove that $X\circ Y$ is a vector lattice. By $\mathbb R^n_{\rm{Lex}}$ we denote the $n$-dimensional vector space $\mathbb R^n$ of all $n$-tuples ordered lexicographically.
Similarly, $\mathbb R^\infty_{\rm{Lex}}$ denotes the space of all real sequences ordered lexicographically. 
Therefore, $\mathbb R\circ \mathbb R$ is the lexicographically ordered real plane $\mathbb R^2_{\rm{Lex}}$, and 
$\mathbb R \circ (\mathbb R \circ \mathbb R)=\mathbb R\circ \mathbb R^2_{\mathrm{Lex}}$ is the lexicographically ordered real vector space $\mathbb R^3$. The positive cones of the these examples fall into the scope of the so-called ``lexicographic cones" studied by Wortel in \cite{Wor19}.

Since the lexicographically ordered real plane $\mathbb R^2_{\rm{Lex}}$ is not Archimedean, $\mathbb R^2_{\rm{Lex}}$ admits a non-order bounded linear functional. As $0\leq ne_2\leq e_1$ for all $n\in \mathbb N$, every order bounded functional $\varphi$ on $\mathbb R^2_{\rm{Lex}}$ satisfies $\varphi(e_2)=0$. Hence, it seems that the order dual of $\mathbb R^2_{\mathrm{Lex}}$ is isomorphic to the first copy of $\mathbb R$ in the lexicographical product. The remaining part of the paper is devoted to determine order duals of the following lexicographically ordered vector lattices $\mathbb R^n_{\rm{Lex}}$ and $\mathbb R^\infty_{\rm{Lex}}$. We start with a more general result.

\begin{theorem}\label{order dual of lex}
For vector lattices $X$ and $Y$ the mapping
$\Phi\colon (X\circ Y)^\sim \to X^\sim$ defined as $\varphi \to \varphi|_X$ is a lattice isomorphism.
\end{theorem}

In particular, if $X^\sim=\{0\}$, then for every vector lattice $Y$ we have $(X\circ Y)^\sim=\{0\}$.

\begin{proof}
We claim that the mapping $\Phi\colon (X\circ Y)^\sim \to X^\sim$ defined as $\Phi(\varphi)=\varphi|_X$ is a lattice isomorphism.
If $\varphi$ is order bounded on $X\circ Y$, then $\Phi(\varphi)$ is order bounded on $X$ since every interval $[a,b]_X$ in $X$ is contained in the interval $[a,b]_{X\circ Y}$ in $X\circ Y$ via the natural identification of $X$ and $X\times\{0\}$.
Clearly, $\Phi$ is a linear operator.

We claim that every order bounded functional $\varphi \in (X\circ Y)^\sim$ is zero on $Y$.
If this were not the case, then there would exist $0<y\in Y$ such that $\varphi(y)\neq 0$. By replacing $\varphi$ with $-\varphi$, if necessary, we may suppose that $\varphi(y)>0$.  Pick any $0\leq x\in X$. Then $0\leq \lambda y\leq x+y$ for all $\lambda\geq 0$ which yields that $\varphi([0,x+y])$ contains  $[0,\infty)$. This is in contradiction with order boundedness of $\varphi$.

If $\Phi(\varphi)=\varphi|_X=0$ for some $\varphi\in (X\circ Y)^\sim$, then $\varphi=0$ as $\varphi|_Y=0$, which proves injectivity of $\Phi$.
Since for every $\varphi\in X^\sim$ we have $\varphi+0\in (X\circ Y)^\sim$ it follows that $\Phi$ is surjective.
To prove that $\Phi$ is a lattice isomorphism, observe that $\varphi+0$ is positive on $X\circ Y$ if and only if $\varphi$ is positive on $X$.
\end{proof}

\begin{corollary}
For $n,m\in \mathbb N$ the order dual of $\mathbb R^n\circ \mathbb R^m$ is lattice isomorphic to $\mathbb R^n$. The lattice isomorphism $\Phi\colon (\mathbb R^n\circ \mathbb R^m)^{\sim} \to \mathbb R^n$ is given as
$$\Phi(\varphi)=(\varphi(e_{1}),\ldots,\varphi(e_{n})).$$
\end{corollary}

\begin{corollary}
For every $n\in \mathbb N\cup \{\infty\}$ the order dual of $\mathbb R^n_{\rm{Lex}}$ is lattice isomorphic to $\mathbb R$.
The lattice isomorphism $\Phi\colon (\mathbb R^n_{\rm{Lex}})^{\sim} \to \mathbb R$ is given as
$\Phi(\varphi)=\varphi(e_1)$.
\end{corollary}

\begin{proof}
Since for $n\in \mathbb N$ we have $\mathbb R^n_{\rm{Lex}}\cong \mathbb R \circ \mathbb R^{n-1}_{\rm{Lex}}$ and $\mathbb R^\infty_{\rm{Lex}}\cong \mathbb R \circ \mathbb R^{\infty}_{\rm{Lex}}$, the order dual of $\mathbb R^n_{\rm{Lex}}$ is lattice isomorphic to the order dual of the first copy of $\mathbb R$ by \Cref{order dual of lex}.
\end{proof}

\subsection*{Acknowledgements}
The first author was supported by the Slovenian Research and Innovation Agency program P1-0222.
The second author was supported by the Slovenian Research and Innovation Agency program P1-0222 and grant J1-50002.

\end{document}